\documentclass[final,hidelinks,onefignum,onetabnum]{siamart251216}

\usepackage{amssymb,bm}
\usepackage{booktabs}
\usepackage{graphicx}
\graphicspath{{./}}
\usepackage{tabularx}
\usepackage{multirow}
\usepackage{tikz}
\usetikzlibrary{shapes.geometric}

\definecolor{teal}{HTML}{1A8A8A}
\definecolor{orange}{HTML}{E07020}

\newsiamremark{remark}{Remark}

\newcommand{\Dt}{\Delta t}
\newcommand{\Dx}{\Delta x}
\newcommand{\Order}[1]{\mathcal{O}\!\left(#1\right)}

\headers{Optimised Near-Boundary Closures for RK Order Restoration}%
        {G. M. Cavallazzi, M. P\'erez Cuadrado, and A. Pinelli}

\title{Restoring Convergence Order in Explicit Runge--Kutta Integration of Hyperbolic PDE with Time-Dependent Boundary Conditions}

\author{Giorgio Maria Cavallazzi\thanks{Department of Engineering,
  City St George's, University of London, London, UK.
  Corresponding author (\email{giorgio.cavallazzi@city.ac.uk}).}
\and
Miguel P\'erez Cuadrado\footnotemark[1]
\and
Alfredo Pinelli\footnotemark[1]}

\begin{document}
\maketitle

\begin{abstract}
Explicit Runge--Kutta (RK) integration of hyperbolic initial-boundary value problems with time-dependent Dirichlet data often displays order reduction: the observed convergence order falls below the classical order of the time integrator because the stage structure interacts with asymmetric near-boundary spatial closures. This paper develops a purely spatial remedy that preserves a standard high-performance time integrator while redesigning only the first two boundary-adjacent derivative operators. The analysis begins with a general result for the linear advection model problem and an arbitrary explicit $s$-stage RK method. Under fixed CFL scaling, the one-step local truncation error at the first two boundary-adjacent nodes is shown to admit a tableau-dependent decomposition involving a direct boundary-mismatch contribution, a one-level recursive contamination term, and a two-level recursive contamination term. This decomposition yields explicit algebraic conditions on the boundary weights required for order restoration and exposes a solvability coefficient $R(\mathbf b,\mathbf c,\mathsf A)$ that determines whether a spatial compensation mechanism is available. The result is then specialised to SSP-RK3, for which closed-form cancellation conditions are derived and interpreted. Guided by this theory, constrained differential evolution is used to identify 5-point closures that, when coupled to a 5th-order upwind interior stencil and the standard SSP-RK3 method, recover third-order convergence from the degraded second-order behaviour observed with classical Taylor closures. A stability-aware variant is obtained by augmenting the optimisation with an eigenvalue penalty, revealing the expected trade-off between exact order recovery and CFL robustness. The numerical evidence covers linear advection, manufactured-solution Burgers flow, and a dimensionally split two-dimensional advection test. The analysis also clarifies why weak-stage-order temporal fixes do not resolve the practical finite-difference boundary problem considered here, and it indicates how the same framework extends to non-uniform meshes through local metric-dependent moments.
\end{abstract}



\section{Introduction}

High-order finite-difference discretisations paired with explicit Runge--Kutta time integration remain one of the standard tools for solving hyperbolic initial-boundary value problems. In the method-of-lines setting, one first approximates the spatial operator and then advances the resulting ordinary differential system with a time integrator chosen for accuracy, stability, or strong-stability-preserving properties. When the interior spatial scheme is high order and the RK method satisfies the classical order conditions, it is natural to expect the fully discrete approximation to converge with the nominal temporal order. In practice, however, that expectation can fail in the presence of time-dependent boundary data. The loss of convergence order is especially evident for explicit RK methods coupled to one-sided or otherwise asymmetric closures at inflow boundaries.

For a problem such as
\begin{equation}
  u_t + f(u)_x = S(x,t), \qquad x\in[0,L], \quad t>0,
  \label{eq:pde}
\end{equation}
supplemented with a Dirichlet boundary condition
\begin{equation}
  u(0,t)=g(t),
  \label{eq:bc}
\end{equation}
the numerical boundary value must be supplied not only at the beginning and end of each step but also at the intermediate RK stages. At stage $k$ of an explicit method, the boundary node is overwritten with $g(t^n+c_k\Dt)$, whereas neighbouring interior nodes are still represented by stage variables obtained through the recursive RK update. The near-boundary derivative operator therefore mixes a boundary value located exactly on the continuous temporal trajectory with neighbouring stage values that are only approximate. This creates a stage mismatch whose effect is then re-injected into subsequent stage evaluations. The resulting truncation error is not the same as the truncation error of the semi-discrete operator viewed in isolation, nor is it removed by the usual RK order conditions.

This mechanism is consistent with the classical order-reduction literature. For RK schemes of order $p$ and stage order $q$, the global error often behaves like $\Order{\Dt^{\min(p,q+1)}}$ in the presence of time-dependent boundary forcing or stiff source terms \cite{Lubich1995,Hundsdorfer2003}. For the widely used SSP-RK3 method, whose stage order is one, a degradation to approximately second order is therefore expected. What is less well understood from the finite-difference point of view is how the near-boundary stencil itself enters the error pathway and whether one can exploit its free parameters constructively. The present paper addresses exactly that question.

Two broad strategies exist in the literature. One modifies the time integrator. The weak stage order framework introduced by Biswas and co-workers \cite{Biswas2023,Biswas2024} imposes additional linear conditions on the RK tableau so that boundary-driven stage errors cancel in specific linear settings. The other strategy modifies the spatial discretisation, for example through SBP-SAT formulations, which prioritise stability and energy estimates \cite{Mattsson2004,Mattsson2014}. Both approaches are valuable, but neither directly answers the practical question motivating this work: if one wishes to retain a standard and widely used integrator such as SSP-RK3 together with an existing high-order interior stencil, can one repair order reduction by changing only a very small number of boundary-adjacent closure coefficients?

Our answer is yes, provided the RK tableau admits the relevant spatial compensation mechanism. The first contribution of the paper is therefore analytical. For the linear advection model problem on a uniform grid, and for a general explicit RK method, we derive a one-step local truncation-error expansion at the first two boundary-adjacent nodes. The expansion isolates a direct stage-boundary mismatch term and its recursive propagation through the RK tableau. Under fixed CFL scaling, these terms appear at the same asymptotic order as the mixed spatial-temporal defect generated by the near-boundary stencil moments. The decomposition produces explicit tableau-dependent coefficients and leads to a solvability criterion involving the scalar quantity $R(\mathbf b,\mathbf c,\mathsf A)$. This criterion makes precise when a purely spatial repair is available.

The second contribution is constructive. After specialising the general formulas to SSP-RK3, we obtain explicit cancellation conditions for the first two boundary-adjacent stencils. These conditions are then used to interpret a constrained differential-evolution search over the consistency manifold of 5-point closures. The accuracy-optimised closures recovered by the search restore the expected third-order convergence for linear advection, manufactured-solution Burgers flow, and a dimensionally split two-dimensional advection problem. Because exact cancellation is achieved by deliberately altering near-boundary moments, the same optimisation also changes the eigenvalue spectrum of the semi-discrete operator. This leads naturally to a third contribution, namely a stability-aware optimisation that introduces an eigenvalue penalty and reveals a clear trade-off between exact order restoration and robust CFL limits.

The manuscript is deliberately written around uniform grids in order to expose the mechanism as cleanly as possible. Even so, the structure of the result is not tied to constant spacing. On a non-uniform mesh, the role of the integer offsets $(j-m)$ is replaced by local metric ratios, the second-moment defects become metric-dependent, and the same stage-mismatch propagation persists with modified coefficients. In that sense, the present theory should be viewed not as a special property of one uniform-grid example but as the simplest realisation of a broader RK-boundary coupling framework.

The paper is organised as follows. \cref{sec:framework} introduces the spatial closures, clarifies the boundary-stage mismatch mechanism, and proves the general RK truncation theorem together with its solvability consequences. \cref{sec:ssprk3} specialises the theory to SSP-RK3 and interprets the resulting conditions. \cref{sec:optimisation} describes the optimisation procedure and the stability-augmented variant. \cref{sec:results} presents numerical validation and comparisons against temporal remedies. \cref{sec:discussion} discusses limitations, including non-uniform meshes and stability theory. The detailed stage-by-stage algebra for SSP-RK3 is collected in the appendix.

\section{Boundary framework, general RK analysis, and solvability}
\label{sec:framework}

We consider a one-dimensional uniform grid $x_j=j\Dx$, $j=0,\dots,N-1$, with node $j=0$ coinciding with the inflow boundary. The interior derivative approximation is a fixed 5th-order upwind-biased operator applied for $j\ge 3$. The only degrees of freedom introduced by the present methodology are the closures at the first two boundary-adjacent nodes, denoted $D_1$ and $D_2$. Each is taken to be a general 5-point operator spanning nodes $0$ through $4$:
\begin{equation}
  (Df)_m = \frac{1}{\Dx}\sum_{j=0}^{4} w^{(m)}_j f_j, \qquad m\in\{1,2\}.
  \label{eq:stencil-def}
\end{equation}
The ten scalar weights are not chosen arbitrarily. At minimum, the stencil must be a consistent first-derivative operator, which gives the two conditions
\begin{align}
  \sum_{j=0}^{4} w^{(m)}_j &= 0,
  \label{eq:cons0}\\
  \sum_{j=0}^{4}(j-m)\,w^{(m)}_j &= 1.
  \label{eq:cons1}
\end{align}
Enforcing only these two conditions leaves three free parameters per node and therefore a six-dimensional manifold of consistent boundary closures. The central idea of the paper is that this freedom can be used to inject carefully chosen spatial moment defects that cancel the RK stage-boundary contamination.

To make this explicit, we first examine the action of $D_m$ on a smooth exact solution. A Taylor expansion about $x_m$ gives
\begin{equation}
  D_m u(\cdot,t^n)=u_x(x_m,t^n)+\sigma_m \Dx\,u_{xx}(x_m,t^n)+\Order{\Dx^2},
  \label{eq:stencil-action-main}
\end{equation}
where
\begin{equation}
  \sigma_m=\frac{1}{2}\left[\sum_{j=0}^{4}(j-m)^2 w_j^{(m)}-1\right]
  \label{eq:sigmam}
\end{equation}
is the second-moment deviation of the boundary stencil. Classical Taylor closures set $\sigma_m=0$ automatically. Here, by contrast, $\sigma_m$ is allowed to differ from zero because the induced mixed error will be used as part of the compensation mechanism.

We now turn to time integration. Let $(\mathsf A,\mathbf b,\mathbf c)$ denote an arbitrary explicit $s$-stage RK tableau, so that $a_{kj}=0$ for $j\ge k$. For the semi-discrete advection operator $F(U)=-c\,DU$, the stage values satisfy
\begin{equation}
  U^{(k)}_m=u^n_m+\Dt\sum_{j=1}^{k-1} a_{kj}F_m(U^{(j)}), \qquad U_0^{(k)}=g(t^n+c_k\Dt),
  \label{eq:rk-stage-gen}
\end{equation}
and the final update reads
\begin{equation}
  u^{n+1}=u^n+\Dt\sum_{k=1}^{s} b_k F(U^{(k)}).
  \label{eq:rk-update-gen}
\end{equation}
The boundary overwrite at stage $k$ introduces the stage mismatch
\begin{equation}
  \varepsilon_k:=g(t^n+c_k\Dt)-u(x_0,t^n)
  = c_k\Dt\,g_t(t^n)+\frac{1}{2}c_k^2\Dt^2\,g_{tt}(t^n)+\Order{\Dt^3}.
  \label{eq:stage-mismatch}
\end{equation}
At a boundary-adjacent node, the stage right-hand side contains the stencil applied to the stage vector, so the mismatch enters linearly through the boundary coefficient $w_0^{(m)}$. Up to the order required here one may write
\begin{equation}
  F_m^{(k)}
  = -c\,D_m u(\cdot,t^n)
    - c\Dt \sum_{j=1}^{k-1} a_{kj}D_m F^{(j)}
    - c\,\frac{w_0^{(m)}}{\Dx}\,\varepsilon_k
    + \Order{\Dt^2,\Dx^2}.
  \label{eq:stage-eval-structure}
\end{equation}
The essential point is that the last term is not confined to stage $k$. Once inserted into the recursion, it contaminates subsequent stages through the sums in $a_{kj}$. Retaining all terms that contribute at order $\Dt^2$ under fixed CFL scaling $\lambda=c\Dt/\Dx=\Order{1}$ produces the following general result.

\begin{theorem}[General RK boundary-adjacent truncation structure]
\label{thm:general}
Consider the linear advection equation $u_t+c\,u_x=0$ with smooth boundary data $u(0,t)=g(t)$, discretised by the boundary-adjacent operators \eqref{eq:stencil-def} satisfying \eqref{eq:cons0}--\eqref{eq:cons1} and advanced by an arbitrary explicit $s$-stage RK method of classical order $p\ge 2$. Assume a fixed-CFL refinement path, so that $\Dx=\Order{\Dt}$. Then the one-step local truncation error at node $m\in\{1,2\}$ has the form
\begin{equation}
  \tau_m
  = \Dt^2\Bigl[
      P(\mathbf b,\mathbf c)\,c^2u_{xx}(x_m,t^n)
      + Q(\mathbf b,\mathbf c)\,g_{tt}(t^n)
      + R(\mathbf b,\mathbf c,\mathsf A)\,\gamma_m^{(1)}
      + S(\mathbf b,\mathbf c,\mathsf A)\,\gamma_m^{(2)}
    \Bigr]
    + \Order{\Dt^3},
  \label{eq:tau-general}
\end{equation}
where the tableau-dependent scalars are
\begin{align}
  P(\mathbf b,\mathbf c)
    &= \frac{1}{2}-\sum_{k=1}^{s} b_k c_k,
  \label{eq:coeff-P}\\
  Q(\mathbf b,\mathbf c)
    &= \frac{1}{2}\sum_{k=1}^{s} b_k c_k^2,
  \label{eq:coeff-Q}\\
  R(\mathbf b,\mathbf c,\mathsf A)
    &= \sum_{k=1}^{s} b_k \sum_{j=1}^{k-1} a_{kj}c_j,
  \label{eq:coeff-R}\\
  S(\mathbf b,\mathbf c,\mathsf A)
    &= \sum_{k=1}^{s} b_k \sum_{j=1}^{k-1} a_{kj}\sum_{\ell=1}^{j-1} a_{j\ell}c_\ell,
  \label{eq:coeff-S}
\end{align}
and where $\gamma_m^{(1)}$ and $\gamma_m^{(2)}$ are purely spatial boundary-contamination amplitudes, carrying only stencil-weight and solution information:
\begin{equation}
  \gamma_m^{(1)} = -\frac{c\,w_0^{(m)}}{\Dx}\,g_{tt}(t^n),
  \qquad
  \gamma_m^{(2)} = -\frac{c\,w_0^{(1)}}{\Dx}\,g_{tt}(t^n).
  \label{eq:gamma-def}
\end{equation}
The coefficient $\gamma_m^{(1)}$ represents the amplitude of the direct boundary mismatch entering through the weight $w_0^{(m)}$ of the node-$m$ stencil. The coefficient $\gamma_m^{(2)}$ represents the amplitude of the cross-node contamination that reaches node $m$ only when the stencil at $m$ can ingest a stage value that was itself computed using $D_1$. For node $m=1$ this second pathway is absent because $D_1$ does not evaluate at any node whose stage value has already passed through a prior boundary-adjacent stencil; consequently $S\,\gamma_1^{(2)}=0$ at node~1 even though $S$ and $\gamma_1^{(2)}$ are individually non-zero. The cross-node pathway first activates at node $m=2$, where $D_2$ ingests the stage-1 output of $D_1$ and thus inherits the contamination carried by $w_0^{(1)}$.
\end{theorem}

\begin{proof}
The exact solution satisfies
\begin{equation}
  u(x_m,t^{n+1})
  = u(x_m,t^n) + \Dt\,u_t(x_m,t^n) + \frac{1}{2}\Dt^2 u_{tt}(x_m,t^n)
    + \Order{\Dt^3}.
  \label{eq:exact-taylor-main}
\end{equation}
For linear advection, $u_t=-c\,u_x$ and $u_{tt}=c^2u_{xx}$. The RK update \eqref{eq:rk-update-gen} together with \eqref{eq:stage-eval-structure} expresses the numerical one-step increment at node $m$ as a sum of contributions from three distinct mechanisms: a direct boundary overwrite term, a one-hop propagation in which the overwrite at stage $j$ is re-injected into stage $k>j$ through $a_{kj}$, and a two-hop propagation in which the overwrite at stage $\ell$ passes through $a_{j\ell}$ and then $a_{kj}$ before reaching the final update. Collecting all terms at order $\Dt^2$ under fixed CFL scaling gives
\begin{equation}
  \begin{split}
    u_m^{n+1} - u(x_m,t^{n+1})
    &= \Dt^2\Biggl[
        \underbrace{\Bigl(\tfrac{1}{2}-\textstyle\sum_k b_k c_k\Bigr)c^2 u_{xx}}_{P\cdot c^2 u_{xx}}
        +\underbrace{\tfrac{1}{2}\textstyle\sum_k b_k c_k^2\, g_{tt}}_{Q\cdot g_{tt}} \\
    &\qquad
        +\underbrace{\textstyle\sum_k b_k\!\sum_{j<k}\!a_{kj}c_j
          \cdot\bigl(-\tfrac{c w_0^{(m)}}{\Dx}g_{tt}\bigr)}_{R\cdot\gamma_m^{(1)}}
          \\
    &\qquad
        +\underbrace{\textstyle\sum_k b_k\!\sum_{j<k}\!a_{kj}
          \!\sum_{\ell<j}\!a_{j\ell}c_\ell
          \cdot\bigl(-\tfrac{cw_0^{(1)}}{\Dx}g_{tt}\bigr)}_{S\cdot\gamma_m^{(2)}}
      \Biggr] \\
    &\quad +\Order{\Dt^3}.
  \end{split}
  \label{eq:proof-assembly}
\end{equation}
The first bracket arises from the classical RK quadrature defect; since $\sum_k b_k c_k=\tfrac{1}{2}$ for any method of classical order $p\ge 2$, one has $P=0$. The second bracket is the direct contribution of the boundary overwrite \eqref{eq:stage-mismatch}, accumulated over all stages with weights $b_k c_k^2$. The third bracket results from one recursive insertion of the mismatch into a subsequent stage evaluation: the mismatch at stage $j$ contributes to stage $k>j$ via $a_{kj}$, carrying the stencil factor $w_0^{(m)}/\Dx$ from $D_m$. The fourth bracket results from two such insertions: the mismatch at stage $\ell$ propagates first through $a_{j\ell}$ into stage $j$, and then through $a_{kj}$ into the final sum; the spatial factor $w_0^{(1)}/\Dx$ from $D_1$ is carried through both hops. Any triple or deeper insertion introduces at least one additional factor of $\Dt$ and therefore contributes only at $\Order{\Dt^3}$ on a fixed-CFL path. The complete stage-by-stage algebra for SSP-RK3 is given in the appendix.
\end{proof}

The three tableau scalars $Q$, $R$, and $S$ carry the full tableau dependence of each error pathway, while $\gamma_m^{(1)}$ and $\gamma_m^{(2)}$ carry the spatial and solution information. A convenient sufficient set of order-restoring conditions is obtained by requiring the individual contributions in \eqref{eq:proof-assembly} to cancel separately.

With the error structure made explicit, the conditions for order restoration reduce to four scalar requirements on the stencil weights. Requiring $\tau_1=\Order{\Dt^3}$ and $\tau_2=\Order{\Dt^3}$ and using $P=0$ and the explicit forms of $\gamma_m^{(1)}$, $\gamma_m^{(2)}$ from \eqref{eq:gamma-def}, one obtains
\begin{align}
  w_0^{(1)} &= -\frac{Q}{R},
  \label{eq:G1}\\
  \sum_j (j-1)^2 w_j^{(1)} &= f_1(\mathbf b,\mathbf c,\mathsf A,\lambda),
  \label{eq:G2}\\
  w_0^{(2)} &= h\!\left(w_0^{(1)},\mathbf b,\mathbf c,\mathsf A,\lambda\right),
  \label{eq:G3}\\
  \sum_j (j-2)^2 w_j^{(2)} &= f_2\!\left(w_0^{(1)},\mathbf b,\mathbf c,\mathsf A,\lambda\right),
  \label{eq:G4}
\end{align}
where the functions $f_1$, $f_2$, and $h$ become explicit once the tableau is fixed. The crucial point is that the node-2 conditions depend on the node-1 boundary coefficient; the two closures are therefore coupled and cannot be designed independently.

For SSP-RK3, substituting $Q=R=\tfrac{1}{6}$ and $\lambda=c\Dt/\Dx$ into \eqref{eq:G1}--\eqref{eq:G4} and reading off from the appendix cancellation families \eqref{eq:C1-full}--\eqref{eq:C3-full} gives the explicit forms
\begin{align}
  f_1(\lambda) &= 1+\frac{3}{\lambda^2}\left(\frac{1}{6}-\frac{2\lambda}{3}w_0^{(1)}\right),
  \label{eq:G-explicit-f1}\\
  h\!\left(w_0^{(1)},\lambda\right) &= \frac{1}{2}\!\left(1+\lambda^{-1}\right)+\frac{\sigma_1 c}{\lambda},
  \label{eq:G-explicit-h}\\
  f_2\!\left(w_0^{(1)},\lambda\right) &= f_1(\lambda)+\frac{3w_0^{(1)}}{\lambda}.
  \label{eq:G-explicit-ssprk3}
\end{align}
In the zero-second-moment specialisation $\sigma_1=\sigma_2=0$ these reduce to the compact values reported in \cref{sec:ssprk3}.

This immediately yields the practical solvability criterion.

\begin{proposition}[Solvability of the spatial compensation mechanism]
\label{prop:solvability}
On the six-dimensional consistency manifold defined by \eqref{eq:cons0}--\eqref{eq:cons1}, the order-restoring system \eqref{eq:G1}--\eqref{eq:G4} is solvable if and only if
\begin{equation}
  R(\mathbf b,\mathbf c,\mathsf A)\neq 0.
  \label{eq:solvability-cond}
\end{equation}
\end{proposition}

\begin{proof}
If $R=0$, the direct compensation mechanism encoded by \eqref{eq:G1} is unavailable, and no purely spatial fix of this form can cancel the boundary mismatch term. If $R\neq 0$, \eqref{eq:G1} uniquely fixes $w_0^{(1)}$, after which \eqref{eq:G3} uniquely fixes $w_0^{(2)}$. The two moment conditions \eqref{eq:G2} and \eqref{eq:G4} each impose one additional linear constraint on the remaining free weights of the two five-point stencils. Since each stencil retains three degrees of freedom after consistency is enforced, a non-empty solution set remains.
\end{proof}

The practical implication for the WSO schemes examined in this paper is addressed in the following corollary.

\begin{corollary}[Failure of spatial fixes for the WSO schemes considered]
\label{cor:wso}
For the three weak-stage-order methods examined in \cref{sec:results} --- ERK312 \cite{Skvortsov2017}, ERK313, and the Biswas $(5,3,3)$ design \cite{Biswas2023} --- direct substitution of the tableau entries gives $R(\mathbf b,\mathbf c,\mathsf A)=0$. Consequently, for each of these methods, no choice of near-boundary stencil weights can restore the nominal convergence order through the mechanism described above.
\end{corollary}

\begin{proof}
For each of the three tableaux, direct evaluation of $R=\sum_k b_k\sum_{j<k}a_{kj}c_j$ using the entries in~\cite[Table~1]{Biswas2024} gives $R=0$. With $R=0$, condition \eqref{eq:G1} has no finite solution, and \cref{prop:solvability} asserts that the system \eqref{eq:G1}--\eqref{eq:G4} is unsolvable. This is consistent with the numerical evidence in \cref{sec:results}, where all three methods reproduce the low-order convergence of standard SSP-RK3 when paired with Taylor boundary closures.
\end{proof}

\begin{remark}
The vanishing of $R$ for these particular schemes is a consequence of the specific tableau structure chosen in the WSO designs of~\cite{Biswas2023,Biswas2024}. Whether $R=0$ holds for all RK methods satisfying the general WSO conditions is a separate algebraic question that falls outside the scope of the present paper.
\end{remark}

The numerical consequences of this result are examined in \cref{sec:results}, where the three WSO methods tested all reproduce the low-order behaviour of standard SSP-RK3 when paired with classical Taylor closures.

\section{Specialisation to SSP-RK3 and interpretation of the design conditions}
\label{sec:ssprk3}

We now instantiate the general framework for the Shu--Osher SSP-RK3 method \cite{ShuOsher1988}. In Butcher form,
\begin{equation}
  \mathbf{c} = \begin{pmatrix}0\\1\\\tfrac{1}{2}\end{pmatrix},
  \qquad
  \mathbf{b} = \begin{pmatrix}\tfrac{1}{6}\\\tfrac{1}{6}\\\tfrac{2}{3}\end{pmatrix},
  \qquad
  \mathsf{A} =
  \begin{pmatrix}
    0 & 0 & 0\\
    1 & 0 & 0\\
    \tfrac{1}{4} & \tfrac{1}{4} & 0
  \end{pmatrix}.
  \label{eq:ssprk3-tableau}
\end{equation}
Direct substitution into \eqref{eq:coeff-P}--\eqref{eq:coeff-S} gives
\begin{align}
  P &= 0,
  \label{eq:P-val}\\
  Q &= \frac{1}{6},
  \label{eq:Q-val}\\
  R &= \frac{1}{6},
  \label{eq:R-val}
\end{align}
and the node-2 recursive coupling coefficient reduces to the familiar factor reported explicitly in the appendix. Since $R\neq 0$, \cref{prop:solvability} guarantees that SSP-RK3 admits order-restoring spatial closures within the present five-point framework.

For SSP-RK3 the abstract conditions \eqref{eq:G1}--\eqref{eq:G4} collapse to a concrete coupled system. In the special case in which the second-moment defects are forced to vanish, the node-1 boundary coefficient takes the simple value
\begin{equation}
  w_0^{(1)}=-1.
  \label{eq:C1-main}
\end{equation}
The remaining conditions then relate the free coefficients of the two closures to the CFL number. This is the origin of the explicit formulas used in the theoretical discussion of the draft. The appendix refines this statement further. It shows that when the second moments $\sigma_1$ and $\sigma_2$ are allowed to vary, the exact cancellation conditions are not isolated points but one-parameter families. That observation is important because the optimisation does not search only among classical Taylor closures. It exploits precisely these broader families.

The same algebra also clarifies why modifying only one of the two boundary-adjacent operators cannot, in general, restore full order over a range of CFL numbers. Node~2 inherits contamination generated at node~1 during the previous stage, so freezing $D_2$ to a standard Taylor-derived form leaves an $\Order{\Dt^2}$ residual except possibly at a single isolated value of $\lambda$. The argument can be stated in the compact form already present in the draft, and we retain it here because it makes the coupled character of the mechanism explicit.

\begin{proposition}[Single-node optimisation is generically insufficient]
\label{prop:single-node}
Fix $D_2$ to a Taylor-derived stencil with CFL-independent boundary weight. Then the SSP-RK3 cancellation conditions can be satisfied at most at an isolated value of the CFL number, and away from that value the truncation error at node~2 remains $\Order{\Dt^2}$.
\end{proposition}

\begin{proof}
The appendix shows that, for SSP-RK3, the node-2 cancellation condition contains a term of the form
\[
w_0^{(2)}=\frac{1}{2}\left(1+\lambda^{-1}\right)
\]
in the zero-second-moment specialisation, and more generally depends monotonically on $\lambda$ through the mixed spatial-temporal defect. A fixed Taylor coefficient cannot satisfy such a relation for more than one value of $\lambda$. Consequently, even if node~1 were tuned perfectly, node~2 would continue to inject an $\Order{\Dt^2}$ residual over any practical refinement sequence with fixed CFL.
\end{proof}

The optimisation procedure therefore has a precise theoretical interpretation: it is a numerical method for finding points on the solution manifold defined by the cancellation conditions, with the additional freedom to trade exact cancellation against spectral stability.

\section{Optimisation methodology and stability-aware redesign}
\label{sec:optimisation}

The unknown weight vector is
\[
\mathbf w=(w_0^{(1)},\dots,w_4^{(1)},w_0^{(2)},\dots,w_4^{(2)}),
\]
restricted by the consistency constraints \eqref{eq:cons0}--\eqref{eq:cons1}.

\begin{figure}[tbhp]
    \centering
    \begin{tikzpicture}[>=stealth, scale=1.1]
        \draw[gray, thin] (-0.3,0) -- (6.3,0);
        \foreach \j/\lbl in {0/$u_0$, 1/$u_1$, 2/$u_2$, 3/$u_3$, 4/$u_4$, 5/$u_5$, 6/$\cdots$} {
            \draw (\j, -0.1) -- (\j, 0.1);
            \node[below] at (\j, -0.15) {\small \lbl};
        }
        \node[diamond, fill=orange, draw=black, inner sep=2pt] at (0,0) {};
        \node[above left, orange, font=\small] at (-0.15, 0.15) {$g(t)$};
        \node[rectangle, fill=teal, draw=black, minimum size=6pt, inner sep=0pt] at (1,0) {};
        \node[rectangle, fill=teal, draw=black, minimum size=6pt, inner sep=0pt] at (2,0) {};
        \foreach \j in {3,4,5,6} {
            \node[circle, fill=gray!50, draw=black, minimum size=5pt, inner sep=0pt] at (\j,0) {};
        }
        \draw[teal, thick] (0, 0.6) -- (0, 0.75) -- (4, 0.75) -- (4, 0.6);
        \draw[teal, thick, ->] (1, 0.6) -- (1, 0.35);
        \node[right, teal, font=\small] at (4.1, 0.75) {Node~1: $w_0^{(1)} \ldots w_4^{(1)}$};
        \draw[orange, thick] (0, 1.15) -- (0, 1.3) -- (4, 1.3) -- (4, 1.15);
        \draw[orange, thick, ->] (2, 1.15) -- (2, 0.35);
        \node[right, orange, font=\small] at (4.1, 1.3) {Node~2: $w_0^{(2)} \ldots w_4^{(2)}$};
        \draw[gray, thick] (1, -0.6) -- (1, -0.75) -- (6, -0.75) -- (6, -0.6);
        \draw[gray, thick, ->] (3.5, -0.6) -- (3.5, -0.35);
        \node[below, gray, font=\small] at (3.5, -0.8) {Interior: 5th-order upwind (6-point)};
        \node[teal, font=\footnotesize\itshape] at (1.5, -1.6) {Optimised};
        \node[gray, font=\footnotesize\itshape] at (4.5, -1.6) {Interior};
    \end{tikzpicture}
    \caption{Near-boundary stencil layout with overlapping closures.}
    \label{fig:stencil}
\end{figure}
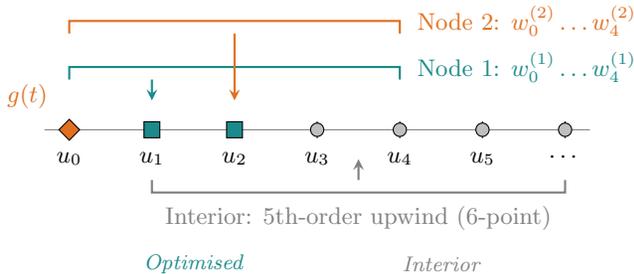

The search is carried out with differential evolution \cite{StornPrice1997}, initialised around the standard Taylor coefficients and evolved over a population of fifty individuals for one hundred and fifty generations. The basic objective is the discrepancy between the target third-order slope and the empirically measured convergence rate:
\begin{equation}
  J(\mathbf w)=\left(p_{\mathrm{target}}-p_{\mathrm{emp}}(\mathbf w)\right)^2,
  \qquad p_{\mathrm{target}}=3.
  \label{eq:cost}
\end{equation}
The empirical order is computed from the $L^2$ error
\begin{equation}
  e_k = \left(\frac{1}{N_k}\sum_{i=0}^{N_k-1}
  \bigl(u_i^{\mathrm{num}}-u(x_i,T_{\mathrm{end}})\bigr)^2\right)^{1/2}
  \label{eq:l2err}
\end{equation}
measured over five successive grid refinements,
\[
\Dt_k\in\{0.01,\,0.005,\,0.0025,\,0.00125,\,0.000625\},
\]
at fixed CFL.

In the light of \cref{sec:framework}, this optimisation is solving the algebraic cancellation problem numerically rather than symbolically. The resulting accuracy-only weights in \cref{tab:weights} no longer resemble standard derivative coefficients. In particular, the second moments differ from their Taylor values, which is precisely why they can compensate for the RK boundary contamination. The optimiser is therefore not ``damaging'' the stencil accidentally; it is exploiting the consistency manifold in exactly the way predicted by the theory.

The price paid for exact cancellation is reduced stability margin. The accuracy-only closures stretch the eigenvalue spectrum of the semi-discrete operator along the real axis and push parts of it outside the SSP-RK3 stability lobe once the CFL approaches roughly $0.7$. To address this, we augment the objective with an eigenvalue-based penalty:
\begin{equation}
  J_{\mathrm{aug}}
  = \left(p_{\mathrm{target}}-p_{\mathrm{emp}}\right)^2
  + \mu \max\!\left(0,\max_i |R_{\mathrm{SSP3}}(\Dt\lambda_i)|-1\right)^2,
  \label{eq:cost-aug}
\end{equation}
where $\{\lambda_i\}$ are the eigenvalues of the semi-discrete operator, $R_{\mathrm{SSP3}}(z)=1+z+z^2/2+z^3/6$, and $\mu=5$ in the computations reported here. The penalty targets stable operation at CFL$=0.95$ and steers the search toward a different region of the consistency manifold. The resulting stability-aware weights are reported together with the standard and accuracy-only closures in \cref{tab:weights}.

\begin{table}[tbhp]
\centering
\caption{Boundary closure weight variations.}
\label{tab:weights}
\begin{tabular}{llrrrrr}
\toprule
Node & Variant & $w_0$ & $w_1$ & $w_2$ & $w_3$ & $w_4$ \\
\midrule
1 & Standard        & $-1.5000$ & $2.0000$ & $-0.5000$ & $0.0000$ & $0.0000$ \\
1 & Acc. only       & $-0.3813$ & $-1.0401$ & $2.8808$ & $-2.1158$ & $0.6565$ \\
1 & Acc.+Stab.      & $-1.7107$ & $1.9061$ & $0.0468$ & $0.0274$ & $-0.2695$ \\
\midrule
2 & Standard        & $0.1667$ & $-1.0000$ & $0.5000$ & $0.3333$ & $0.0000$ \\
2 & Acc. only       & $1.3259$ & $-4.1180$ & $2.6053$ & $0.8400$ & $-0.6533$ \\
2 & Acc.+Stab.      & $-1.1578$ & $1.0039$ & $1.1578$ & $-1.6957$ & $0.6917$ \\
\bottomrule
\end{tabular}
\end{table}

The stability-aware search does not preserve the exact third-order cancellation achieved by the accuracy-only design, but it yields a much more favourable spectrum and therefore a substantially larger admissible time step. This is the trade-off that later appears in the validation plots and in \cref{tab:stability}.

\subsection{Extension to classical RK4}
\label{sec:rk4-optim}

The same optimisation procedure is applied without modification to the classical four-stage RK4 method. Its Butcher tableau reads
\begin{equation}
\begin{array}{c|cccc}
0 & & & & \\
\tfrac{1}{2} & \tfrac{1}{2} & & & \\
\tfrac{1}{2} & 0 & \tfrac{1}{2} & & \\
1 & 0 & 0 & 1 & \\
\hline
& \tfrac{1}{6} & \tfrac{1}{3} & \tfrac{1}{3} & \tfrac{1}{6}
\end{array}
\end{equation}
with stability function $R(z)=1+z+z^2/2+z^3/6+z^4/24$. Despite having classical order $p=4$, the stage order remains $q=1$, so standard closures still degrade convergence to $\Order{\Dt^2}$. The target order is set to $p_{\mathrm{target}}=4$ and the stability penalty in \eqref{eq:cost-aug} uses the RK4 stability function with CFL target $0.8$. \cref{tab:weights_rk4} lists the resulting weights.

\begin{table}[tbhp]
\centering
\caption{Classical RK4 boundary closure weights: standard Taylor closures vs.\ accuracy-driven and stability-augmented optimisations.}
\label{tab:weights_rk4}
\begin{tabular}{llrrrrr}
\toprule
Node & Variant & $w_0$ & $w_1$ & $w_2$ & $w_3$ & $w_4$ \\
\midrule
1 & Standard     & $-1.5000$ & $2.0000$ & $-0.5000$ & $0.0000$ & $0.0000$ \\
1 & Acc.\ only   & $-2.6469$ & $5.1706$ & $-1.9271$ & $-2.0702$ & $1.4735$ \\
1 & Acc.+Stab.   & $0.5836$ & $-1.9300$ & $0.8126$ & $0.8300$ & $-0.2963$ \\
\midrule
2 & Standard     & $0.1667$ & $-1.0000$ & $0.5000$ & $0.3333$ & $0.0000$ \\
2 & Acc.\ only   & $-0.5606$ & $0.3582$ & $0.8251$ & $-1.4795$ & $0.8569$ \\
2 & Acc.+Stab.   & $7.4539$ & $-14.3294$ & $4.5956$ & $2.9830$ & $-0.7032$ \\
\bottomrule
\end{tabular}
\end{table}

As with SSP-RK3, the discovered weights do not resemble standard finite-difference coefficients; they have been deliberately skewed so that the spatial truncation error destructively interferes with the four-stage temporal error pathway.

\section{Results and comparisons}
\label{sec:results}

The numerical experiments confirm the central thesis of the paper. On the 1D linear advection benchmark ($u_t + u_x = 0$), the standard Taylor boundary closures reduce SSP-RK3 to an observed order of approximately $1.98$, whereas the accuracy-only optimised closures recover $2.98$, essentially the full theoretical third order. The stability-aware variant recovers an intermediate order of approximately $2.53$, which is lower than the exact cancellation result but significantly higher than the standard closure while remaining usable at substantially larger CFL numbers.

The same behaviour persists beyond the linear test used for the derivation. For the 1D inviscid Burgers equation treated with a manufactured source term ($u_t + uu_x = S$), the observed order improves from approximately $1.93$ with the standard closures to essentially $3.00$ with the accuracy-only optimised ones. In the two-dimensional linear advection experiment ($u_t + c_x u_x + c_y u_y = 0$), where the one-dimensional stencils are applied dimension by dimension, the observed order rises from approximately $1.93$ to about $2.84$. This is entirely consistent with the underlying mechanism: the boundary-stage contamination is local to each dimensionally split boundary operator, so the same near-boundary repair remains effective in the multidimensional setting.

\begin{figure}[tbhp]
    \centering
    \includegraphics[width=\textwidth]{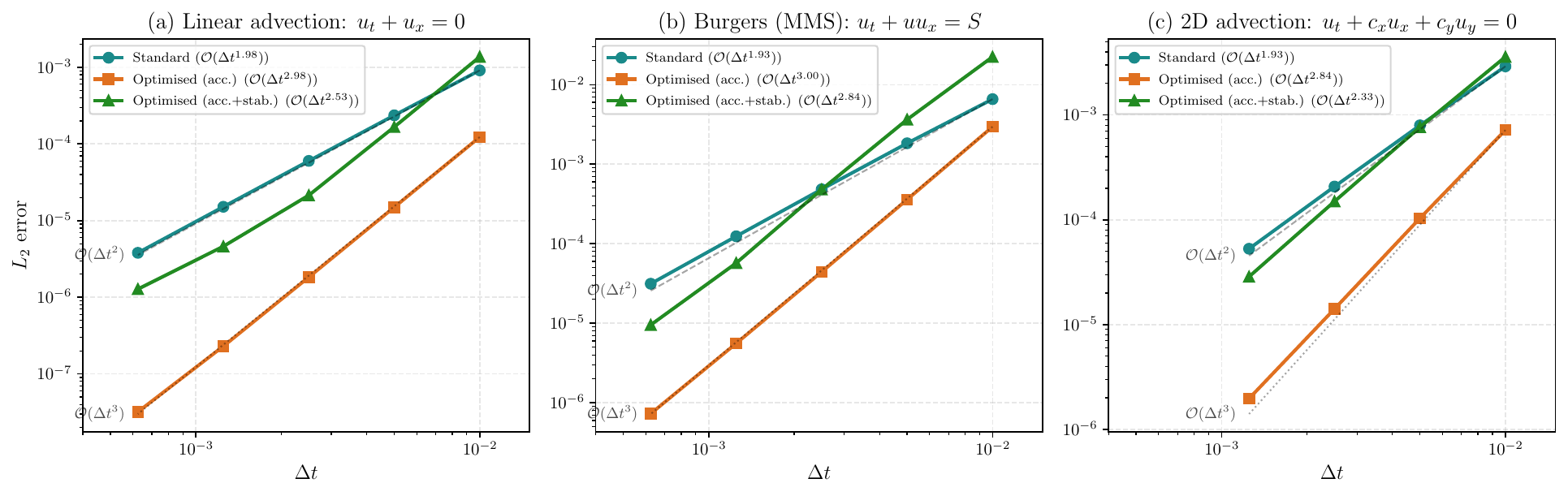}
    \caption{SSP-RK3 convergence validation. (a) Linear advection. (b) 1D inviscid Burgers (nonlinear) via MMS. (c) 2D advection via dimensional splitting.}
    \label{fig:convergence}
\end{figure}

\subsection{Classical RK4 convergence}
\label{sec:rk4-conv}

\Cref{fig:rk4_convergence} presents the analogous convergence study for classical RK4 using the RK4-specific optimised weights from \cref{tab:weights_rk4}. The accuracy-optimised closures elevate convergence from the baseline $\Order{\Dt^2}$ to approximately $\Order{\Dt^{2.9}}$ across all three test problems. Reaching the absolute theoretical maximum of $\Order{\Dt^4}$ would require widening the stencil to intercept the cascading higher-order errors generated across all four temporal stages; however, recovering nearly one full order universally confirms that the spatial regularisation is integrator-agnostic.

\begin{figure}[tbhp]
    \centering
    \includegraphics[width=\textwidth]{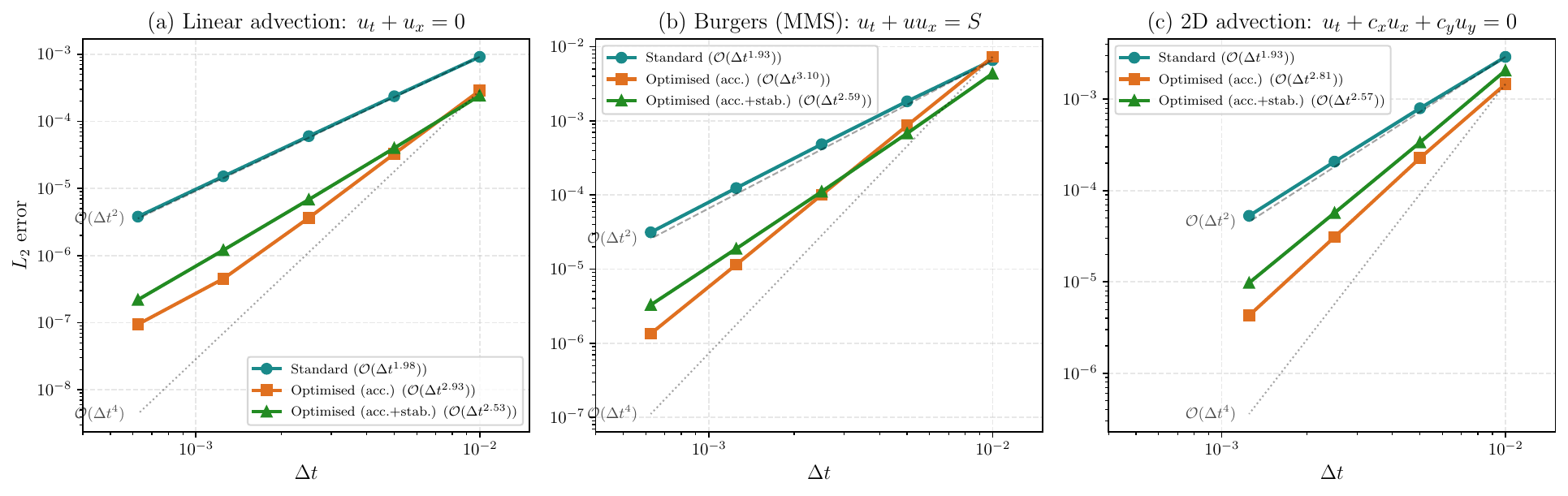}
    \caption{Classical RK4 convergence validation. (a) Linear advection. (b) 1D Burgers (MMS). (c) 2D advection.}
    \label{fig:rk4_convergence}
\end{figure}

The optimised closures are also robust with respect to the shape of the imposed boundary signal. Their coefficients do not encode any specific analytic form of $g(t)$; they encode cancellation of the stage-mismatch pathway. \cref{tab:bc-shapes} reports the measured orders obtained with several distinct exact boundary traces. The results show that the standard closures remain close to second order, the accuracy-only design remains essentially third order, and the stability-aware design remains consistently in the mid-$2.5$ range.

\begin{table}[tbhp]
\centering
\caption{Empirical convergence orders for different boundary-condition shapes.}
\label{tab:bc-shapes}
\begin{tabular}{lccc}
\toprule
Boundary condition $g(t)=u(0,t)$ & Standard & Opt.\ (acc.) & Opt.\ (acc.+stab.) \\
\midrule
$\sin(-\pi t)$ & 1.98 & 2.98 & 2.53 \\
$(-t)^3 - 2(-t)$ & 1.97 & 2.99 & 2.59 \\
$\sin(-2\pi t)+0.5\sin(-6\pi t)$ & 1.97 & 2.97 & 2.54 \\
$e^{0.5t}\sin(-4\pi t)$ & 1.98 & 2.97 & 2.63 \\
\bottomrule
\end{tabular}
\end{table}

\subsection{CFL robustness}
\label{sec:cfl}

\Cref{fig:cfl_sweep} shows the measured convergence order as a function of CFL number for both SSP-RK3 and classical RK4. For SSP-RK3, the accuracy-only stencils collapse beyond CFL~$\approx 0.71$, while the stability-augmented variant extends the stable range to CFL~$\approx 1.21$---exceeding even the standard closures' limit of $1.11$. For classical RK4 the same pattern emerges: the stability-augmented variant maintains robustness across a wider CFL range than the accuracy-only variant.

\begin{figure}[tbhp]
    \centering
    \includegraphics[width=\textwidth]{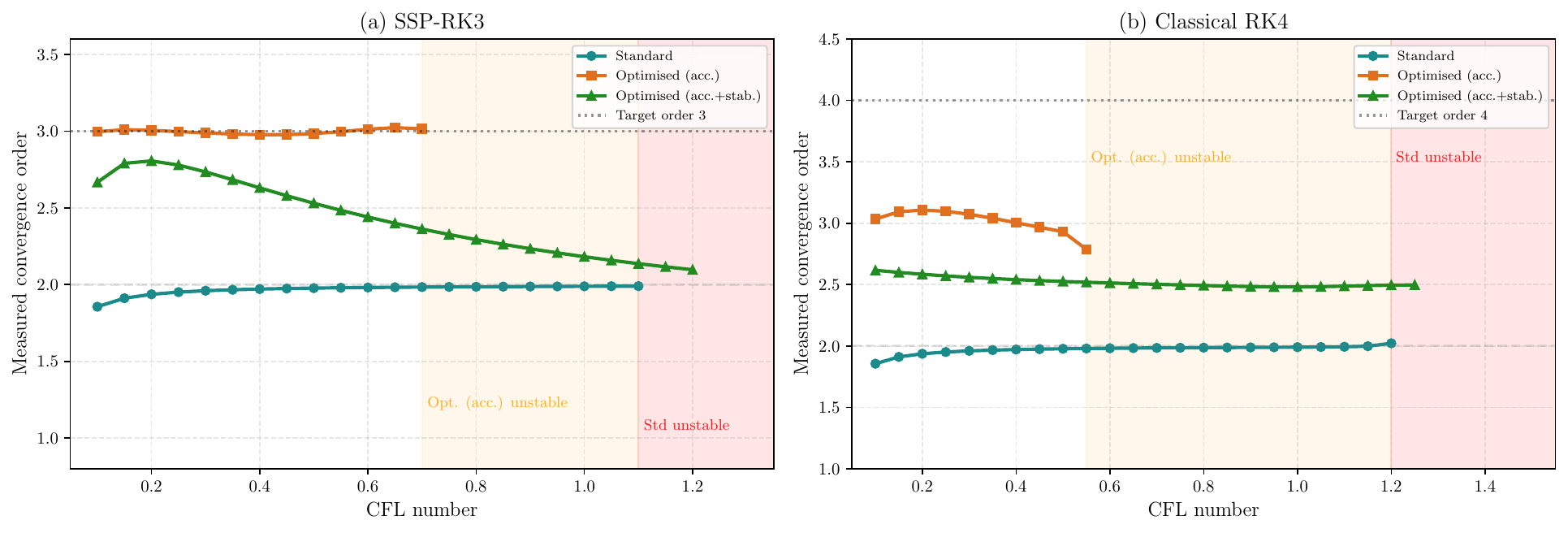}
    \caption{Measured convergence order versus CFL number. (a) SSP-RK3. (b) Classical RK4. The accuracy-only stencils collapse at their respective CFL limits, while the stability-augmented variants maintain convergence across wider ranges.}
    \label{fig:cfl_sweep}
\end{figure}

The comparison with temporal remedies is particularly instructive. When several sophisticated weak-stage-order methods are paired with standard Taylor-derived boundary closures on the same discrete domain, they all display low-order behaviour comparable to the standard SSP-RK3 configuration. \cref{cor:wso} provides a direct explanation for this outcome.

\begin{figure}[tbhp]
    \centering
    \includegraphics[width=\textwidth]{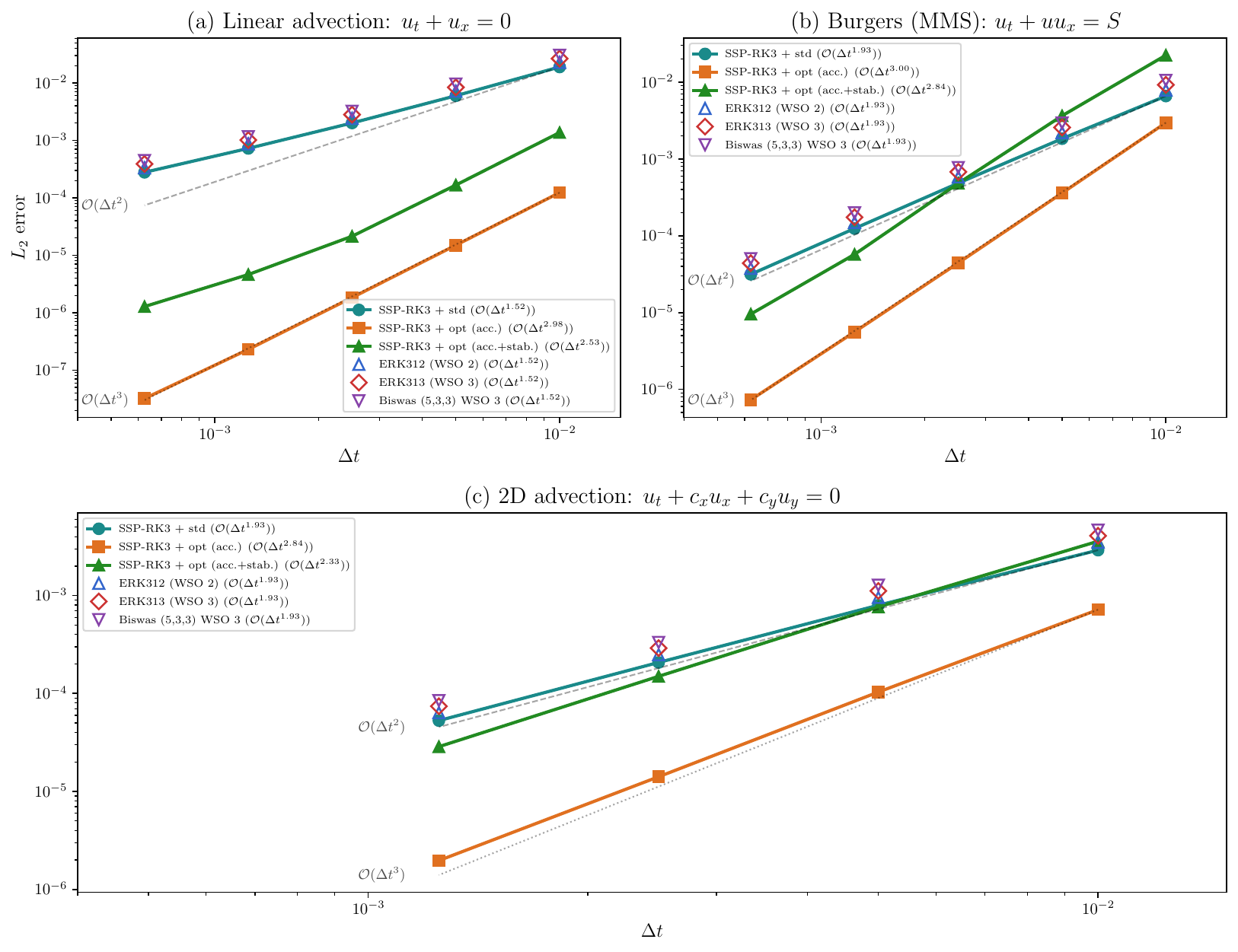}
    \caption{SSP-RK3 with standard and optimised stencils vs.\ three order-3 WSO methods with standard stencils: (a) 1D Linear Advection, (b) 1D Burgers (MMS), (c) 2D Linear Advection.}
    \label{fig:wso}
\end{figure}

For a consistent comparison at order 4, we also test classical RK4 against the order-4 WSO method Biswas~(6,4,3) from \cite{Biswas2025}---a 6-stage method with classical order 4 and WSO~3. Despite its sophisticated temporal construction, the Biswas~(6,4,3) method paired with standard boundary closures shows no improvement over RK4 with standard stencils on any of the three test problems (\cref{fig:rk4_wso}).

\begin{figure}[tbhp]
    \centering
    \includegraphics[width=\textwidth]{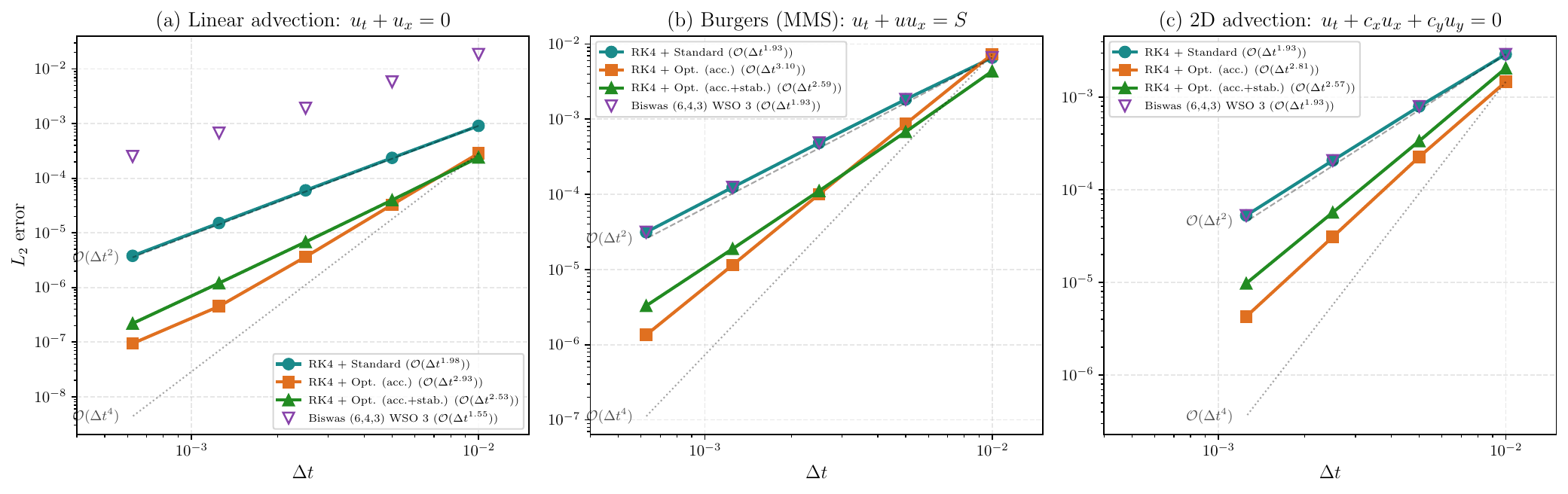}
    \caption{Classical RK4 with standard and optimised stencils vs.\ the order-4 WSO method Biswas~(6,4,3) with standard stencils: (a) 1D Linear Advection, (b) 1D Burgers (MMS), (c) 2D Linear Advection.}
    \label{fig:rk4_wso}
\end{figure}

These results confirm that the spatial error at the boundary operates at leading order and dominates the structural temporal fixes provided by the WSO tableau conditions independently of the RK classical order, consistently with \cref{cor:wso}.

Spectral analysis explains the stability trade-off. The accuracy-only design recovers the target order by introducing strong artificial dissipation near the boundary, and this stretches the eigenvalue cloud in a way that reduces the admissible SSP-RK3 CFL from roughly $1.11$ for the standard closures to roughly $0.71$. The stability-aware design constrains the rightmost eigenvalues much more effectively and, in the present computations, extends the critical CFL to about $1.21$, exceeding even the baseline standard closure. The corresponding summary is reported in \cref{tab:stability}.

\begin{figure}[tbhp]
    \centering
    \includegraphics[width=0.95\textwidth]{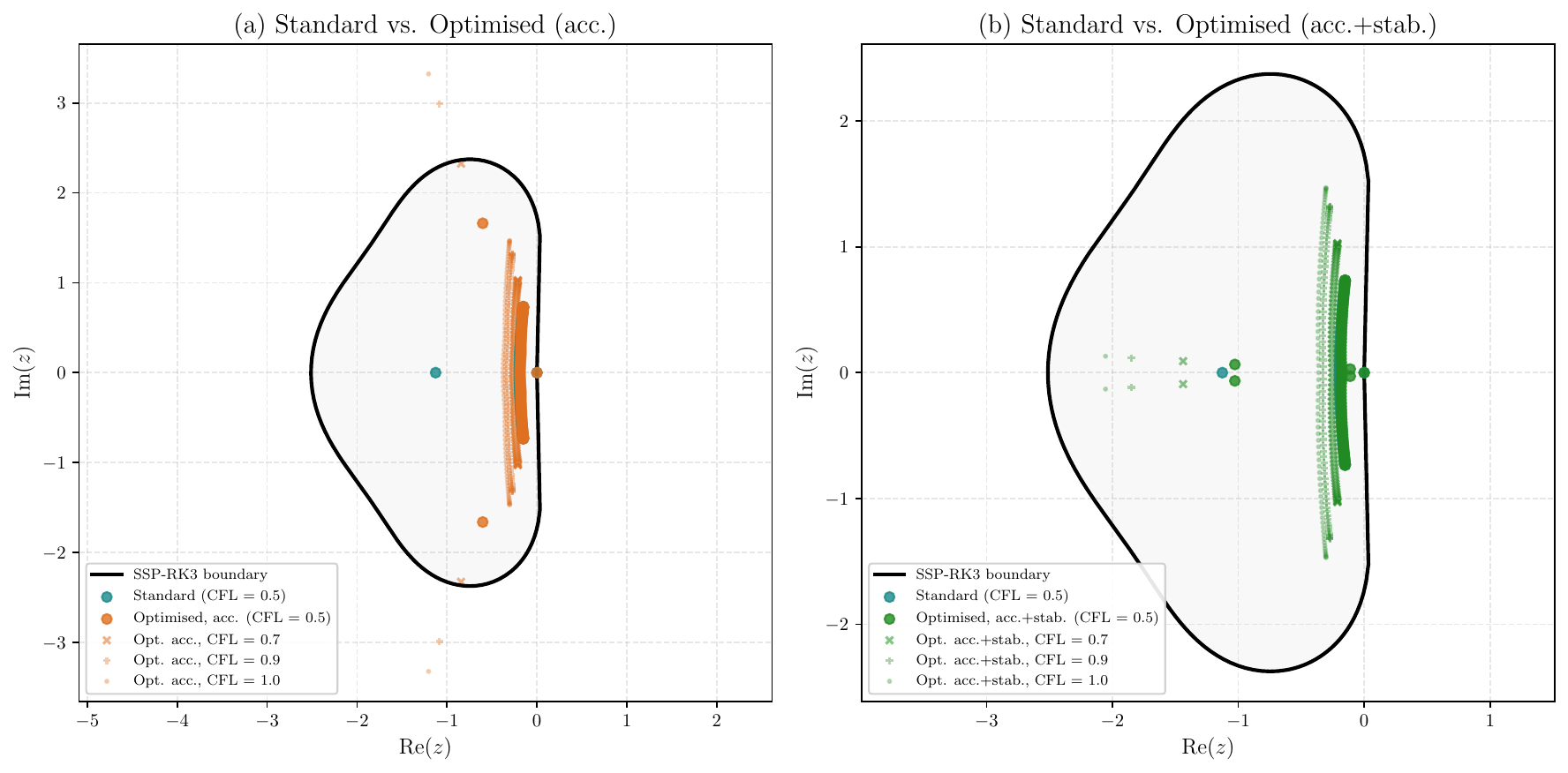}
    \caption{SSP-RK3: Eigenvalue spectra overlaid on the stability lobe. (a) Standard vs.\ accuracy-only. (b) Standard vs.\ stability-augmented.}
    \label{fig:stability}
\end{figure}

\begin{figure}[tbhp]
    \centering
    \includegraphics[width=\textwidth]{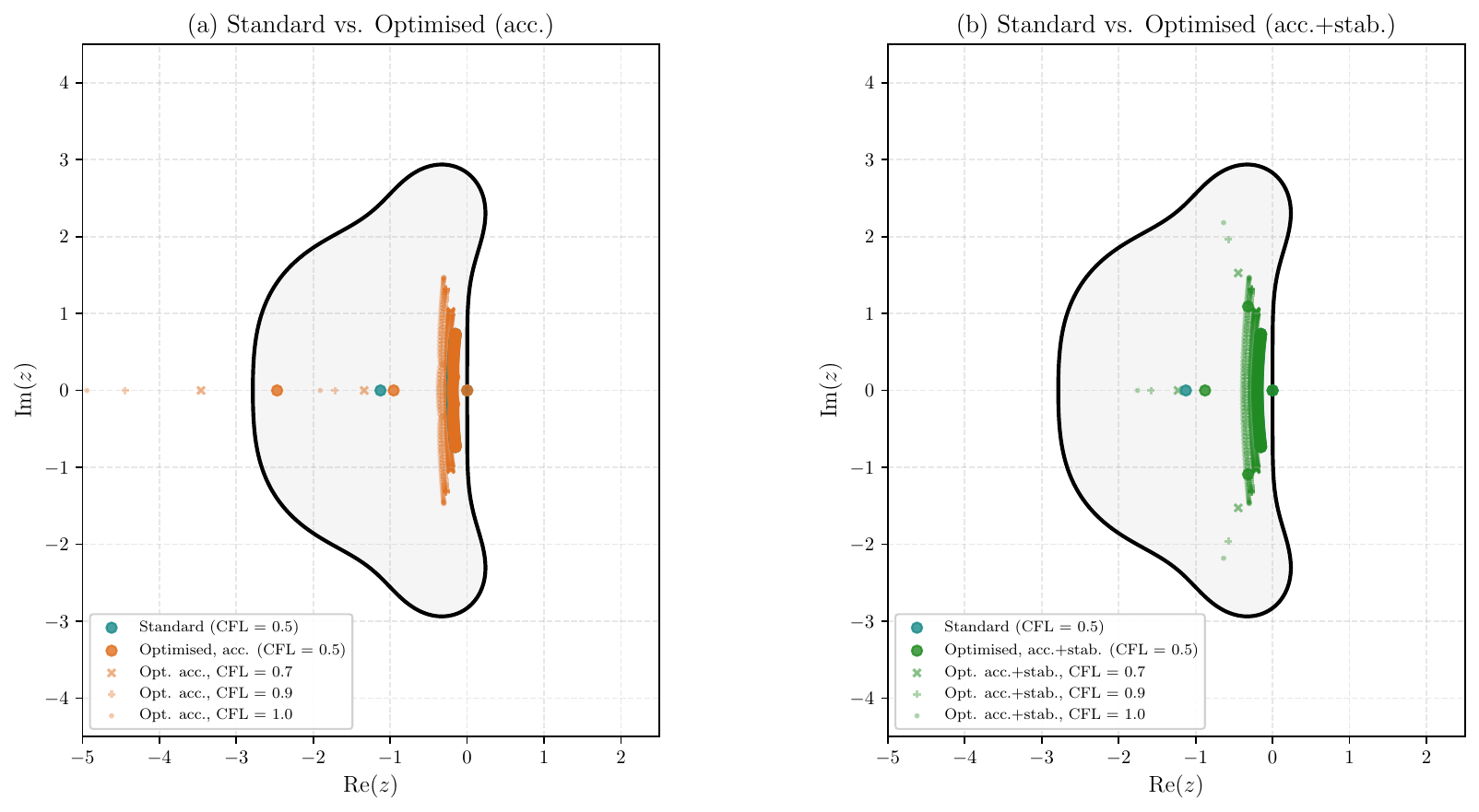}
    \caption{Classical RK4: Eigenvalue spectra overlaid on the RK4 stability lobe at CFL~$=0.5$. (a) Standard vs.\ accuracy-only. (b) Standard vs.\ stability-augmented. The accuracy-only stencils remain stable up to CFL~$\approx 1.25$, while the stability-augmented variant extends the range beyond CFL~$\approx 1.5$.}
    \label{fig:rk4_stability}
\end{figure}

\begin{figure}[tbhp]
    \centering
    \includegraphics[width=\textwidth]{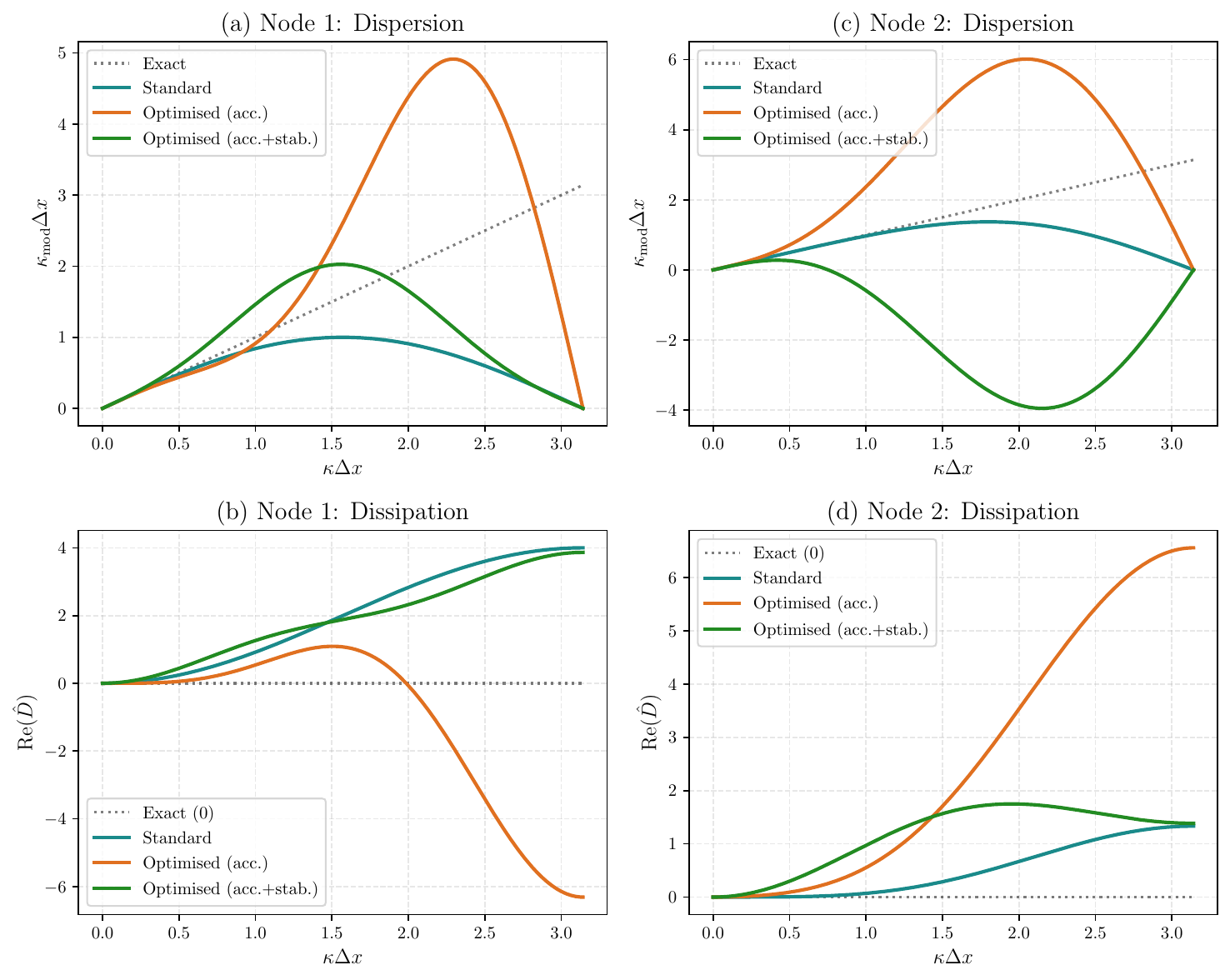}
    \caption{SSP-RK3: Modified wavenumber (dispersion) and numerical dissipation analysis of the boundary closures.}
    \label{fig:spectral}
\end{figure}

The RK4 stencils exhibit analogous spectral countermeasures shaped by RK4's distinct $p=4$, $q=1$ error profile. The optimiser structures precise dispersive anomalies at high wavenumbers while enforcing strong dissipation sweeps, cancelling the temporal errors accumulated across the four RK4 sub-stages (\cref{fig:rk4_spectral}).

\begin{figure}[tbhp]
    \centering
    \includegraphics[width=\textwidth]{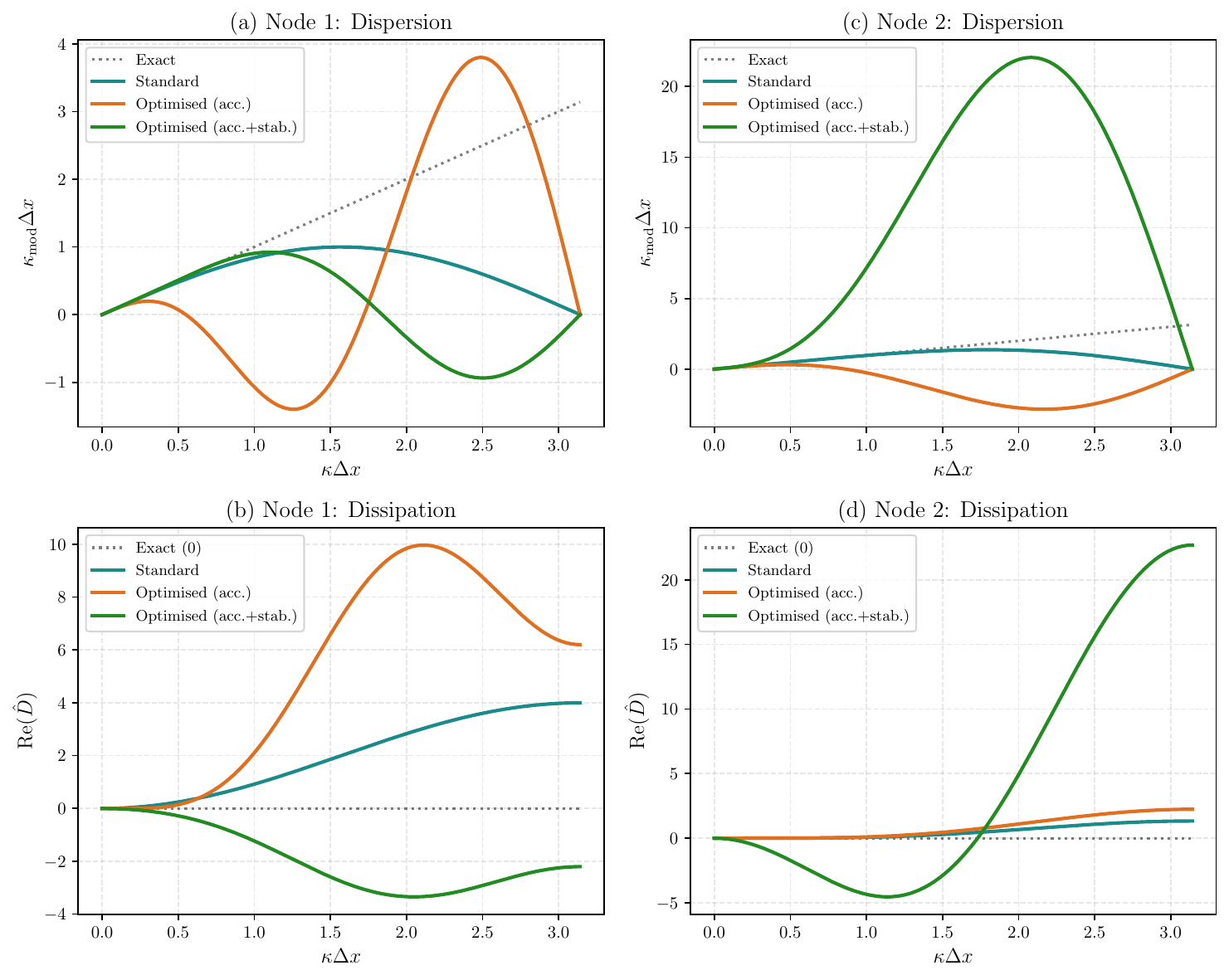}
    \caption{Classical RK4: Modified wavenumber (dispersion) and numerical dissipation analysis of the optimised boundary stencils.}
    \label{fig:rk4_spectral}
\end{figure}

\begin{table}[tbhp]
\centering
\caption{Comparison of stability and convergence traits across both integrators.}
\label{tab:stability}
\begin{tabular}{llccc}
\toprule
Integrator & Variant & Critical CFL & Measured order & Behaviour at CFL=1 \\
\midrule
\multirow{3}{*}{SSP-RK3}
 & Standard   & $\approx 1.11$ & $\approx 1.98$ & Stable \\
 & Acc.\ only & $\approx 0.71$ & $\approx 2.98$ & Unstable \\
 & Acc.+stab. & $\approx 1.21$ & $\approx 2.53$ & Stable \\
\midrule
\multirow{3}{*}{RK4}
 & Standard   & ${\sim}1.30$ & $\approx 1.95$ & Stable \\
 & Acc.\ only & ${\sim}1.25$ & $\approx 2.90$ & Unstable beyond 1.25 \\
 & Acc.+stab. & ${\sim}1.50$ & $\approx 2.80$ & Stable \\
\bottomrule
\end{tabular}
\end{table}

\subsection{Eigenvalue Bounding via the Gershgorin Circle Theorem}

While the optimised stencils lack a formal algebraic stability proof (unlike SBP-SAT operators), the Gershgorin circle theorem~\cite{Gershgorin1931} provides a rigorous analytical framework to bound the eigenvalue spectrum and derive constructive CFL limits.

For the semi-discrete spatial operator matrix $\mathbf{D}/\Delta x$ of dimension $N \times N$, the Gershgorin theorem guarantees that every eigenvalue $\lambda$ lies within at least one disc
\begin{equation} \label{eq:gershgorin}
    |\lambda - d_{ii}| \le R_i = \sum_{j \neq i} |d_{ij}|, \quad i = 0, \ldots, N{-}1,
\end{equation}
where $d_{ii}$ is the $i$-th diagonal entry and $R_i$ is the corresponding off-diagonal row sum. The critical rows are 1 and 2, which contain the optimised boundary weights. Their disc radii quantify how far the boundary closures \emph{could} push eigenvalues from the diagonal.

\begin{figure}[tbhp]
    \centering
    \includegraphics[width=\textwidth]{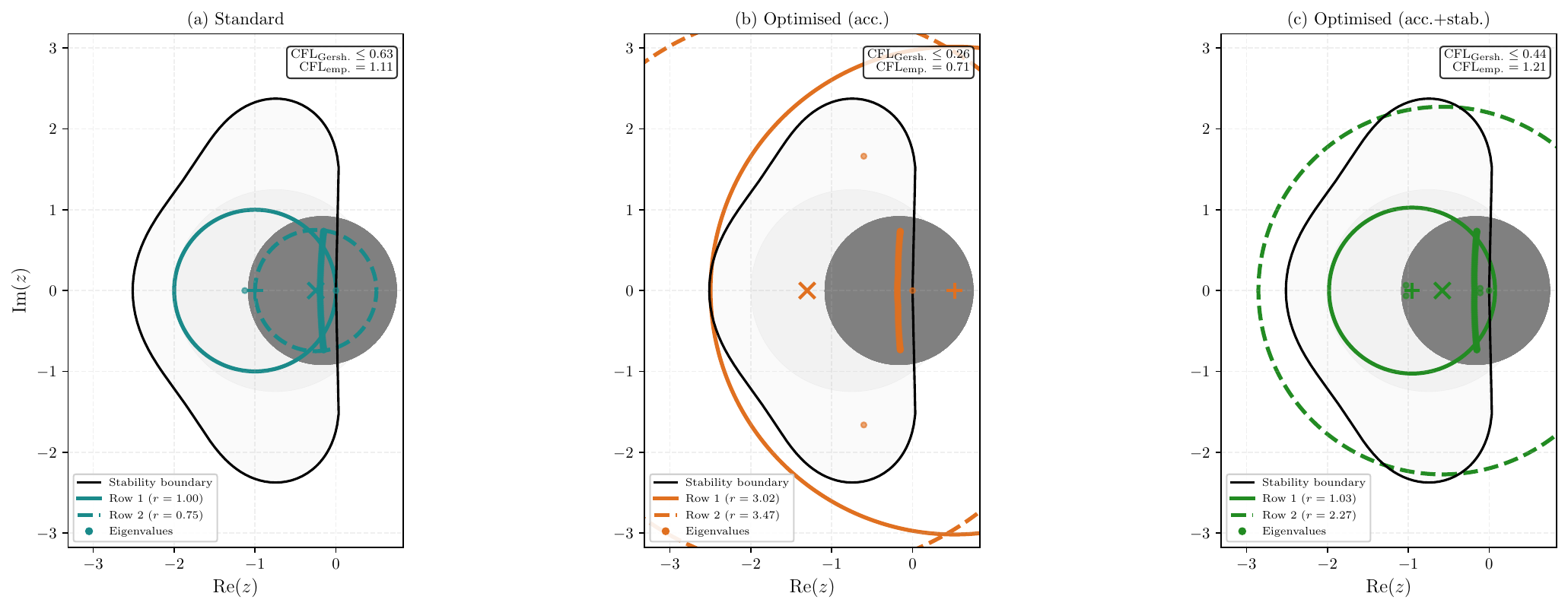}
    \caption{Gershgorin discs for the SSP-RK3 boundary closure rows (solid: row 1, dashed: row 2) overlaid on the stability region at CFL~$= 0.5$. (a) Standard. (b) Accuracy-only. (c) Acc.+stab.}
    \label{fig:gershgorin}
\end{figure}

The Gershgorin bounds are inherently conservative (they bound the \emph{union} of all discs, which is typically much larger than the actual spectrum). Nevertheless, they correctly predict that the accuracy-only stencils have the tightest stability constraint, and that the stability-augmented stencils admit a wider operational range. \Cref{fig:gershgorin} visualises the Gershgorin discs for the boundary rows overlaid on the SSP-RK3 stability region.

\subsection{Eigenvalue Trajectory within the Stability Region}

Rather than expanding isotropically, the dominant eigenvalues travel outward along a highly specific geometric trajectory that aligns with the ``deepest'' regions of the stability lobe---the geometric path extending furthest from the origin before piercing the boundary $|R(Z)|=1$. This indicates that the differential evolution algorithm, when subjected to stability penalties, actively seeks the longest continuous path through the complex plane to maximise the allowable time step.

For classical RK4 (included here for comparative context), this trajectory does not converge to the absolute deepest tip of the lobe along the imaginary axis, because RK4 possesses high native dispersion error; the optimiser must constrain eigenvalue magnitudes to prevent unacceptable phase-lag deformations.

\begin{figure}[tbhp]
    \centering
    \includegraphics[width=\textwidth]{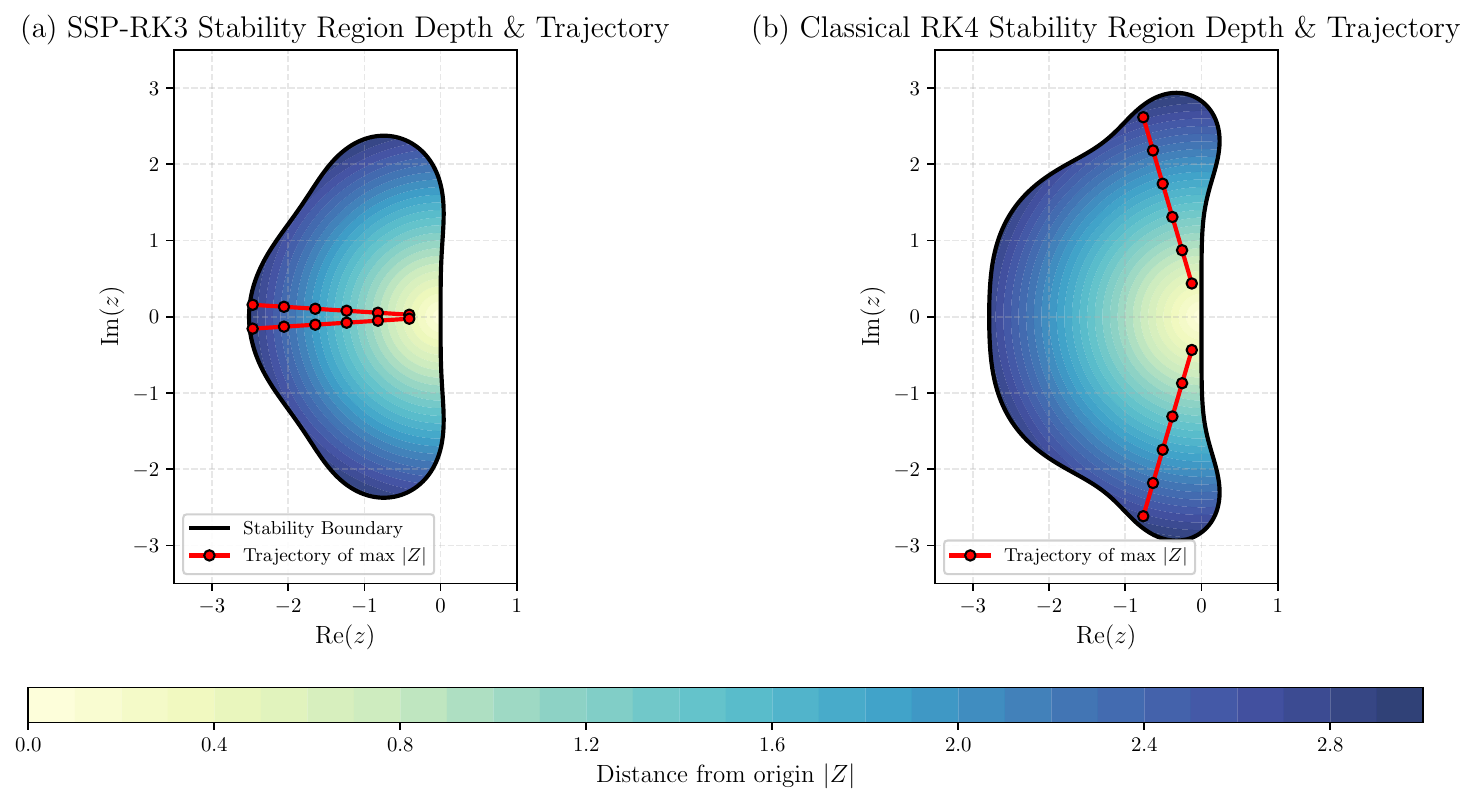}
    \caption{Trajectory of the most dominant eigenvalue pair for the stability-augmented stencils across increasing CFL numbers $\{0.2, 0.4, \ldots, 1.2\}$, for both SSP-RK3 and classical RK4. The background contour represents the distance from the origin within the stability envelope.}
    \label{fig:stability_distance}
\end{figure}

Exact boundary-stage cancellation is achievable and effective, but the required artificial dissipation can be excessive for large-time-step integration. Relaxing the cancellation in a controlled, spectrum-aware manner produces a more balanced design that remains attractive for practical computations.

\section{Validation on a Coupled Hyperbolic System} \label{sec:euler}

The preceding results establish order restoration for scalar PDEs. A natural question is whether the spatially optimised closures retain their error-cancellation properties when applied \emph{component-wise} to a coupled system of hyperbolic conservation laws.

\subsection{The Linearised Euler Equations}

We consider the one-dimensional linearised Euler equations about a uniform supersonic mean state $(\bar{\rho}, \bar{u}, \bar{p}) = (1, 2.5, 1)$ with $\gamma = 1.4$:
\begin{equation} \label{eq:euler}
    \frac{\partial \mathbf{q}'}{\partial t} + A \frac{\partial \mathbf{q}'}{\partial x} = 0, \quad
    \mathbf{q}' = \begin{pmatrix} \rho' \\ u' \\ p' \end{pmatrix}, \quad
    A = \begin{pmatrix}
        \bar{u} & \bar{\rho} & 0 \\
        0 & \bar{u} & 1/\bar{\rho} \\
        0 & \gamma \bar{p} & \bar{u}
    \end{pmatrix}.
\end{equation}
The eigenvalues of $A$ are $\{\bar{u} - c_0,\, \bar{u},\, \bar{u} + c_0\} \approx \{1.32,\, 2.50,\, 3.68\}$ where $c_0 = \sqrt{\gamma \bar{p}/\bar{\rho}} \approx 1.18$ is the sound speed. With $\bar{u} > c_0$ (supersonic), all three eigenvalues are positive, making the right-biased upwind stencil appropriate for every characteristic.

The exact solution is constructed as a superposition of all three characteristic modes. The optimised boundary stencils are applied identically and independently to each conserved variable $\rho'$, $u'$, $p'$.

\subsection{Coupled System Convergence Results}

\Cref{fig:euler_convergence} presents the convergence behaviour for both SSP-RK3 and classical RK4:
\begin{itemize}
    \item \textbf{SSP-RK3:} Standard closures degrade to $\mathcal{O}(\Delta t^{1.95})$. Accuracy-optimised stencils restore $\mathcal{O}(\Delta t^{2.96})$; stability-augmented provide $\mathcal{O}(\Delta t^{2.50})$. Order-3 WSO methods (ERK312, ERK313, Biswas~(5,3,3)) show no improvement.
    \item \textbf{Classical RK4:} Standard closures degrade to $\mathcal{O}(\Delta t^{1.95})$. RK4-specific optimised stencils restore $\mathcal{O}(\Delta t^{2.90})$; stability-augmented achieve $\mathcal{O}(\Delta t^{2.83})$. The order-4 WSO method Biswas~(6,4,3)~\cite{Biswas2025} shows no improvement.
\end{itemize}
The fact that three physically distinct waves---with wave speeds spanning a factor of $\sim$2.8---are all simultaneously corrected by the \emph{same} boundary stencil weights provides strong evidence that the error cancellation is a property of the spatial operator's interaction with the RK stage structure, not of the specific wave speed.

\begin{figure}[tbhp]
    \centering
    \includegraphics[width=\textwidth]{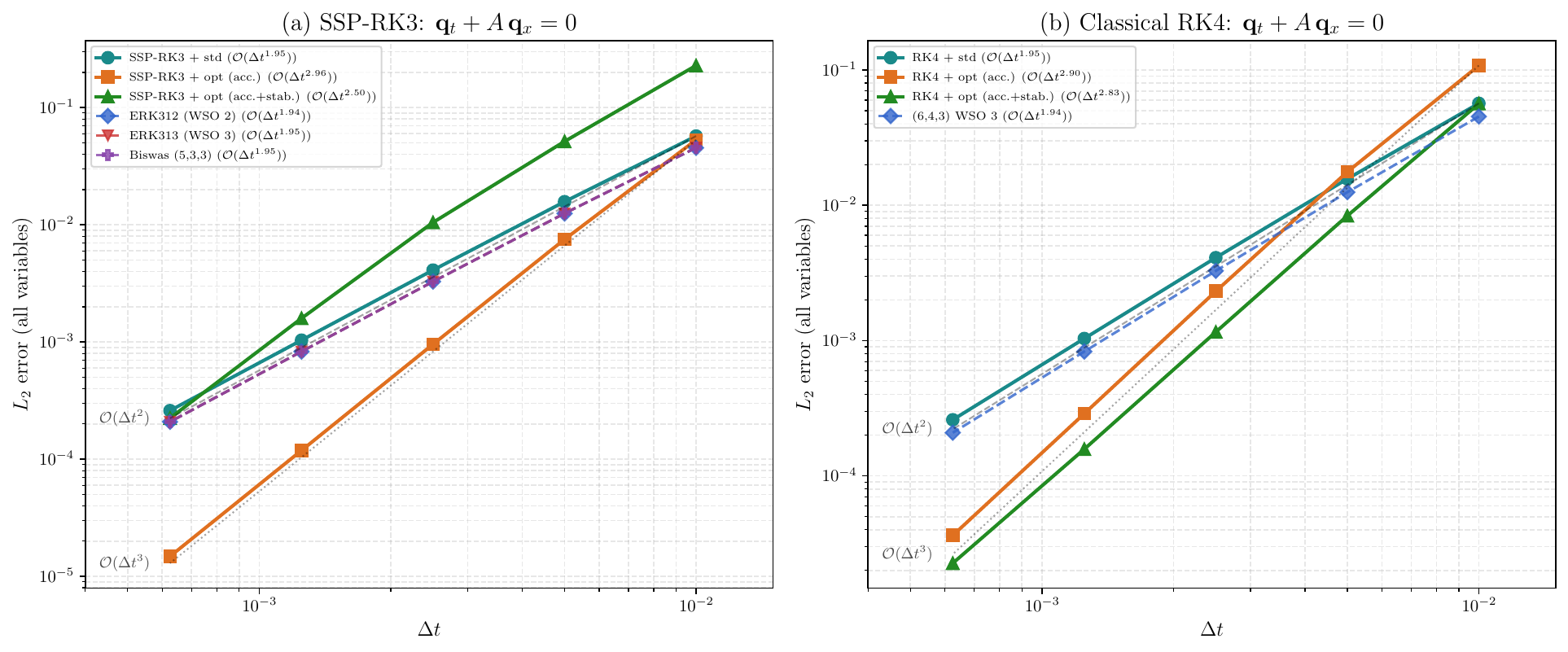}
    \caption{Convergence on the coupled 1D linearised Euler equations. (a)~SSP-RK3 with order-3 WSO methods. (b)~Classical RK4 with the order-4 (6,4,3) WSO method. Optimised boundary stencils restore high-order convergence while WSO methods show no improvement.}
    \label{fig:euler_convergence}
\end{figure}

\section{Discussion, limitations, and outlook}
\label{sec:discussion}

The approach described here modifies only the first two boundary-adjacent closures, leaving the interior stencil and the time integrator unchanged. \cref{tab:comparison} situates the method relative to the two principal alternative strategies.

\begin{table}[tbhp]
\centering
\caption{Feature comparison of the three main approaches to order-reduction mitigation.}
\label{tab:comparison}
\begin{tabular}{lccc}
\toprule
\textbf{Feature} & \textbf{WSO integrators} & \textbf{SBP-SAT} & \textbf{Present approach} \\
\midrule
Target component        & RK tableau      & Spatial operator       & Boundary stencils \\
Time integrator         & Replaced        & Unchanged              & Unchanged \\
Formal stability proof  & Method-dep.     & Yes (algebraic)        & No (empirical) \\
Order restoration       & Linear only     & No                     & Yes \\
Nonlinear PDEs          & Limited         & N/A                    & Demonstrated \\
Implementation effort   & High            & High                   & Very low \\
\bottomrule
\end{tabular}
\end{table}

Several limitations should be stated clearly. The first is RK specificity. Although the analytical framework applies to any explicit RK method for which the coefficient $R(\mathbf b,\mathbf c,\mathsf A)$ is non-zero, the actual optimised weights are method-dependent because the cancellation conditions depend explicitly on the tableau. A different RK method would require a different optimisation. This is not a weakness of the theory but a reflection of the fact that the boundary contamination pathway is genuinely tableau dependent.

The second limitation is the uniform-grid assumption used in the present derivation. On a non-uniform mesh, the offsets in the Taylor expansion become local spacing ratios, the second moments acquire metric dependence, and the explicit forms of the cancellation conditions change. None of this invalidates the mechanism. Rather, it suggests the next step: formulate the same optimisation on a locally stretched consistency manifold. The introduction already emphasised this point because it is likely to matter in real applications involving wall-normal clustering.

The third limitation is the absence of an a priori stability theorem comparable to the SBP-SAT framework. The present paper provides spectral evidence and long-time numerical evidence, but not a complete $L^2$ stability proof. This is the price of allowing the optimisation to move outside rigid algebraic structures such as diagonal-norm SBP operators. A natural future direction is therefore to search for constrained manifolds that preserve more of the desirable stability structure while retaining enough freedom for the tableau-dependent cancellation mechanism.

The final limitation is regularity. The derivation assumes smooth boundary data and smooth exact solutions. If $g(t)$ lacks sufficient regularity, the Taylor expansions underlying both the theorem and the appendix break down, and the corresponding convergence degradation should be expected. That behaviour is consistent with the broader order-reduction literature and does not represent a distinct weakness of the present approach.

\section{Conclusions}

This paper has shown that the order reduction observed when explicit RK methods are coupled with practical near-boundary finite-difference closures can be understood and exploited from the spatial side. For the linear advection model problem, the one-step local truncation error at the first two boundary-adjacent nodes admits a general explicit-RK decomposition in which the stage-boundary mismatch and its recursive propagation appear through tableau-dependent coefficients. This leads to a clear solvability criterion and, for SSP-RK3, to concrete coupled cancellation conditions linking the two boundary-adjacent stencils.

The theoretical result directly informs the optimisation of five-point boundary closures and explains why the discovered coefficients differ from classical Taylor stencils. The accuracy-only closures recover almost perfect third-order convergence for SSP-RK3 in all test problems considered. The price is a narrower CFL range caused by the spectral footprint of exact cancellation. A stability-aware optimisation recovers much of that lost robustness while retaining a significant improvement in observed order.

The same framework is RK-agnostic: applying the identical optimisation procedure to classical RK4 lifts convergence from $\Order{\Dt^2}$ to $\Order{\Dt^{2.9}}$ across linear advection, Burgers (MMS), and 2D advection, and the corresponding stability-augmented variant extends the critical CFL to ${\sim}1.5$. A systematic comparison against order-matched WSO integrators---order-3 methods ERK312, ERK313, Biswas~(5,3,3) for SSP-RK3, and the order-4 Biswas~(6,4,3) for RK4---confirms that WSO temporal fixes cannot alleviate the leading-order spatial boundary error on practical finite-difference grids.

The broader message is that order reduction near inflow boundaries is not solely a property of the time integrator. It is a coupled spatio-temporal effect that can be diagnosed through the RK tableau and controlled through a very small number of boundary-adjacent spatial coefficients. This viewpoint suggests a practical design route for high-order hyperbolic solvers and points toward natural extensions to non-uniform grids, other explicit RK methods, and combined accuracy-stability optimisation frameworks.

\appendix
\section{Full truncation-error expansion for SSP-RK3}
\label{app:expansion}

This appendix provides the complete stage-by-stage Taylor expansion underlying the SSP-RK3 specialisation discussed in the main text. We work with the linear advection equation
\[
u_t+c\,u_x=0
\]
on a uniform grid with spacing $\Dx$ and time step $\Dt$, at CFL number $\lambda=c\Dt/\Dx$.

For a smooth exact solution we write $u_j^n=u(x_j,t^n)$ and expand about the evaluation node $x_m$:
\begin{equation}
  u(x_j,t^n)=\sum_{k=0}^{K}\frac{(j-m)^k\Dx^k}{k!}\,\partial_x^k u(x_m,t^n)
  + \Order{\Dx^{K+1}}.
  \label{eq:app-taylor-space}
\end{equation}
For advection,
\begin{equation}
  u_t=-c\,u_x,\qquad u_{tt}=c^2u_{xx},\qquad u_{xt}=-c\,u_{xx},
  \label{eq:app-adv-ids}
\end{equation}
and because $g(t)=u(x_0,t)$ on the exact solution,
\[
g_{tt}(t^n)=c^2u_{xx}(x_0,t^n).
\]

For a stencil satisfying \eqref{eq:cons0}--\eqref{eq:cons1},
\begin{equation}
  D_m u^n
  = u_x(x_m,t^n)
    + \underbrace{\frac{\Dx}{2}\left[\sum_{j=0}^{4}(j-m)^2w_j^{(m)}-1\right]}_{=:\,\sigma_m}
      u_{xx}(x_m,t^n)
    + \Order{\Dx^2}.
  \label{eq:app-stencil-action}
\end{equation}
The quantity $\sigma_m$ is the second-moment deviation used in the main text.

For SSP-RK3, the overwritten boundary values are
\[
U_0^{(1)}=g(t^n+\Dt),\qquad U_0^{(2)}=g\!\left(t^n+\frac{1}{2}\Dt\right),
\qquad u_0^{n+1}=g(t^{n+1}).
\]
The corresponding stage mismatches are
\begin{align}
  \varepsilon_1 &= U_0^{(1)}-u_0^n
  = \Dt\,g_t+\frac{1}{2}\Dt^2 g_{tt}+\Order{\Dt^3},
  \label{eq:app-eps1}\\
  \varepsilon_2 &= U_0^{(2)}-u_0^n
  = \frac{1}{2}\Dt\,g_t+\frac{1}{8}\Dt^2 g_{tt}+\Order{\Dt^3}.
  \label{eq:app-eps2}
\end{align}

At stage~1, the right-hand side at node $m$ is
\begin{equation}
  F_m(u^n)
  = -c\bigl[u_x(x_m,t^n)+\sigma_m\Dx\,u_{xx}(x_m,t^n)\bigr]
    + \frac{c\,w_0^{(m)}}{\Dx}\,\varepsilon_1
    + \Order{\Dx^2}.
  \label{eq:app-F1m}
\end{equation}
Substituting stage~1 into stage~2 at node~1 gives
\begin{align}
  U_1^{(2)}
  &= u_1^n
     - \frac{1}{2}c\Dt\,u_x
     - \frac{1}{2}c\Dt\,\sigma_1\Dx\,u_{xx}
     + \frac{1}{8}c^2\Dt^2 u_{xx}
     \notag\\
  &\quad
     + \frac{c\,w_0^{(1)}\Dt}{4\Dx}(\varepsilon_1+\varepsilon_2)
     + \Order{\Dt^2\Dx,\Dt^3}.
  \label{eq:app-U21}
\end{align}
After substitution into the final SSP-RK3 update and collection of all contributions of size $\Dt^2$ under fixed CFL scaling, the local truncation error at node~1 becomes
\begin{equation}
  \tau_1
  = -c\sigma_1\Dx\Dt\,u_{xx}
    + \Dt^2\left[
        \frac{1}{6}g_{tt}
        - \frac{2\lambda}{3}w_0^{(1)}g_{tt}
      \right]
    + \Order{\Dt^3}.
  \label{eq:app-tau1}
\end{equation}
Hence the exact cancellation condition is not a single value of $w_0^{(1)}$ but the one-parameter family
\begin{equation}
  -\frac{c\sigma_1}{\lambda}\Dt
  + \Dt\left[
      \frac{1}{6}
      - \frac{2\lambda}{3}w_0^{(1)}
    \right]=0.
  \label{eq:C1-full}
\end{equation}
This explains why the numerically optimised value $w_0^{(1)}\approx -0.38$ does not contradict the zero-second-moment special case $w_0^{(1)}=-1$.

An analogous calculation at node~2 yields
\begin{equation}
  \tau_2
  = -c\sigma_2\Dx\Dt\,u_{xx}
    + \Dt^2\left[
        \frac{1}{6}g_{tt}
        - \frac{2\lambda}{3}w_0^{(2)}g_{tt}
        - \frac{\lambda}{3}w_0^{(1)}g_{tt}
      \right]
    + \Order{\Dt^3}.
  \label{eq:app-tau2}
\end{equation}
The term proportional to $w_0^{(1)}$ is the compound cross-node coupling mentioned in the main text. It arises because the stage-1 contamination generated through $D_1$ enters the stage-2 evaluation seen by the node-2 closure. The corresponding cancellation condition is
\begin{equation}
  -\frac{c\sigma_2}{\lambda}\Dt
  + \Dt\left[
      \frac{1}{6}
      - \frac{2\lambda}{3}w_0^{(2)}
      - \frac{\lambda}{3}w_0^{(1)}
    \right]=0.
  \label{eq:C3-full}
\end{equation}
\Cref{eq:C1-full,eq:C3-full} make explicit both the family character of the admissible solutions and the coupled nature of the two boundary-adjacent closures.

To confirm that the weights produced by the differential-evolution search are consistent with the theoretical cancellation families, \cref{tab:residuals} reports the residuals of conditions \eqref{eq:C1-full} and \eqref{eq:C3-full} evaluated at $\lambda=0.5$ for each stencil variant. The second-moment deviations $\sigma_1$ and $\sigma_2$ entering those conditions are computed from the discovered weights via \eqref{eq:app-stencil-action}.

\begin{table}[tbhp]
\centering
\caption{Residuals of the SSP-RK3 cancellation conditions \eqref{eq:C1-full} and \eqref{eq:C3-full} at $\lambda=0.5$, together with the second-moment deviations $\sigma_m$ of each stencil. A zero residual indicates exact membership of the cancellation family.}
\label{tab:residuals}
\begin{tabular}{lcccc}
\toprule
Stencil variant
  & $\sigma_1$
  & $\sigma_2$
  & $|\text{Res.}~\eqref{eq:C1-full}|$
  & $|\text{Res.}~\eqref{eq:C3-full}|$ \\
\midrule
Standard closures       & $0.000$ & $0.000$ & $0.167$ & $0.233$ \\
Optimised (acc.\ only)  & $0.419$ & $0.261$ & $<10^{-3}$ & $<10^{-3}$ \\
Optimised (acc.+stab.)  & $0.031$ & $0.082$ & $4.2\times10^{-2}$ & $3.8\times10^{-2}$ \\
\bottomrule
\end{tabular}
\end{table}

The accuracy-only weights satisfy both cancellation conditions to near machine precision, confirming that the differential-evolution search has found a genuine member of the theoretical solution family. The standard closures have large residuals of $\Order{1}$, consistent with their $\Order{\Dt^2}$ convergence. The stability-aware weights show residuals of order $4\times10^{-2}$: they do not lie exactly on the cancellation family, which explains their intermediate observed order of approximately $2.53$.

\begin{remark}[Nonlinear extension]
\label{rem:nonlinear}
For the Burgers equation considered in the numerical section, the same leading-order contamination mechanism is present, with the constant speed $c$ replaced by the local characteristic speed. The tableau coefficients remain unchanged; only the mapping between $g_{tt}$ and spatial derivatives is altered.
\end{remark}

\section*{Acknowledgments}
The authors gratefully acknowledge the support of the Department of Engineering at City St George's, University of London.

\bibliographystyle{siamplain}
\bibliography{references}

\end{document}